\documentclass[letterpaper, 10 pt, conference]{ieeeconf}  
\usepackage[T1]{fontenc}
\usepackage[utf8]{inputenc}
\usepackage{float}
\usepackage{color,xcolor}
\usepackage{verbatim}

\IEEEoverridecommandlockouts                              

\newtheorem{theorem}{Theorem}
\let\labelindent\relax
\usepackage{enumitem}

\usepackage{amsmath,amssymb,graphicx,mathptmx}

\title{\LARGE \bf
A Bilevel Market Equilibrium Model for Natural Gas Pipeline Network Operations: Convex Relaxations and Single-level Reformulations
}

\author{Novel Kumar Dey$^{1}$, Afrooz Jalilzadeh$^{1}$, Erfan Yazdandoost Hamedani$^{1}$, Saif R. Kazi$^{2}$ and Anatoly Zlotnik$^{2}$
\thanks{*Research presented in this article was supported by the Laboratory Directed Research and Development program of Los Alamos National Laboratory under Exploratory Research project number 20261039ER. Research conducted at Los Alamos National Laboratory is done under the auspices of the National Nuclear Security Administration of the U.S. Department of Energy under Contract No. 89233218CNA0000001. Los Alamos National Laboratory report number LA-UR-26-29038.}
\thanks{$^{1}$Department of Applied Mathematics (Novel Dey)
and Systems and Industrial Engineering
(Afrooz Jalilzadeh and Erfan Yazdandoost Hamedani),
University of Arizona, Tucson, USA.
{\tt\small \{ndey, afrooz, erfany\}@arizona.edu}}%
\thanks{$^{2}$Applied Mathematics \& Plasma Physics, Theoretical Division, Los Alamos National Laboratory, Los Alamos, USA.
{\tt\small \{skazi, azlotnik\}@lanl.gov}}}

\begin{document}

\maketitle
\thispagestyle{empty}
\pagestyle{empty}

\begin{abstract}
We propose an integrated market equilibrium model for determining the spot market prices and flow rates in a natural gas pipeline network. The traditional approach solves the market equilibrium model separately from the natural gas network flow allocation problem in a sequential manner. Our proposed framework synthesizes marketing and operations by using a bilevel optimization formulation where the market equilibrium model in the upper level is connected and simultaneously solved with the natural gas network flow model in the lower level. We introduce convex relaxations for the nonconvex Weymouth pipe pressure drop equation and a surrogate for the compressor operating cost. Furthermore, we compare multiple single level reformulations of the bilevel problem based on Karush-Kuhn-Tucker (KKT) optimality conditions, Wolfe duality, and Lagrange duality conditions. Numerical studies on a small 6-node network demonstrates the efficacy and accuracy of our model and the comparison between different duality and optimality based reformulations.

\end{abstract}


\section{INTRODUCTION}
The urgent need to increase electricity supplies worldwide is driving  rapid construction of natural gas-fired generators to service base load in the power grid as well as to compensate for variability and intermittency in concurrently expanding renewable electricity sources \cite{jenkins2015dynamic}.  The resulting interdependence between energy flows in gas pipelines and the power grid, as well as markets for wholesale electricity and natural gas has motivated methods to coordinate operations \cite{he2018coordination,zlotnik2016coordinated} as well as financial markets for wholesale power and gas in day-ahead and in real-time markets \cite{schwele2021coordination}. Gas-electric system interdependence results in complex market dynamics \cite{jenkins2015dynamic}, in which actions in one market influence outcomes in both.  


A major difference between how the power grid and gas pipelines function is the degree to which physical and financial decision-making are integrated.  For the power grid, system operators dispatch generators based on the outcomes of optimal power flow (OPF) optimization problems \cite{dommel1968optimal}.  Generation asset managers provide structured bids to the wholesale electricity market administrator, and the market is cleared day-ahead and in near real-time using optimization.  The OPF problem determines the power production setpoints for participating generators that meet demand, satisfy transmission constraints, and minimize cost.  The dual variables are interpreted as locational marginal prices of power \cite{caramanis1982optimal}.  In contrast, gas pipelines are built to provide a given transportation capacity on each segment, and gas transportation service is sold to shippers \cite{mohlin2021us,hulshof2016market}.  This first stage typically uses capacity models, and does not consider physical variables such as pressure or flow.  Pipeline operators then determine compressor settings and the flow allocation on the pipeline that satisfy consumer nominations.  


Researchers have examined problem formulations to synthesize financial and operating processes in gas pipeline networks using optimization \cite{oneill87,read12,rudkevich17hicss}.  A proposed solution is the optimal gas flow (OGF) problem, which has been examined in steady-state \cite{wu2017adaptive,baker2026bling} and dynamic contexts \cite{zlotnik2019optimal,mak2019dynamic}.  While it can be used to produce physical and financial solutions, the OGF is a single-level nonlinear program (NLP) formulation that additively includes competing terms that maximize revenue and minimize compressor operating costs in the objective function. While the OGF can be interpreted as a representation of an efficient market for congestion pricing on a pipeline network, it does not represent the actual two-stage function of wholesale gas markets and operations.  In addition, the nonlinearity and nonconvexity of gas flow equations warrants careful interpretation of the dual solution in the OGF and other optimization problems for natural gas networks.

To model the hierarchical interplay between economic trading and physical transport in liberalized gas markets, recent studies established multi-level equilibrium frameworks where upper-level market participants optimize competitive choices while lower-level transmission system operators (TSOs) enforce physical network constraints \cite{Plein2021Bilevel,Grimm2025Tractable}. Because embedded lower-level physics that are governed by non-convex equations like the Weymouth pressure-drop relationship preclude standard Karush--Kuhn--Tucker (KKT) or strong-duality reformulations, concurrent studies introduced tailored convex relaxations (such as second-order cone bounds) coupled with disjunctive cuts and branch-and-bound techniques \cite{Schewe2022Global,Heitsch2023Convex}. This integration of physical transport dynamics with convex relaxation schemes yields computationally tractable single-level reformulations that rigorously account for TSO control, market power, and network bottlenecks in large-scale natural gas grids.

In this study, we develop a bi-level optimization formulation that represents the wholesale market and operating solution stages of decision making related to flow scheduling in gas pipeline systems.  The upper or outer level is a capacity model that maximizes economic value for pipeline users based on supply and demand bids, and accounts for regulated tariffs on each pipeline segment.  The lower or inner level solves an optimal gas flow problem that minimizes the cost of gas compressor operation to deliver the cleared gas nominations subject to the physical and engineering constraints used in the OGF.  We then apply and compare various convex relaxations and single-level reformulations, including Karush-Kuhn-Tucker (KKT) based mathematical program with complementarity constraints (MPCC) \cite{dempe2020bilevel}, Wolfe-duality (WDP) \cite{li2023wolfe}, Mond-Weir (MDP) \cite{li2024mondweir}, and Lagrange-dual \cite{dempe2025duality} reformulations.  Our key contribution is a theorem on projection of feasible sets and lifting of feasible points for these reformulations, and we use a market equilibrium test case for a small synthetic network to compare the latter using several performance metrics.


The remainder of the paper is as follows. In Section \ref{sec:bilevel}, we describe the bilevel market equilibrium problem expansion of the OGF.  In Section \ref{sec:relax}, we describe convexification of the equality constraints that represent steady-state gas flow through a pipe.  In Section \ref{sec:reform}, we describe single-level reformulations of the bilevel problem, and then state and prove our main feasibility equivalence theorem.  We describe the test case and numerical experiments in Section \ref{sec:numer}, and then conclude in Section \ref{sec:conc}.

\section{Bilevel Gas Network Market Model}\label{sec:bilevel}
The Natural Gas market equilibrium flow model is formulated as a bilevel optimization problem. The upper or outer level optimization represents the market equilibrium problem with a cost based economic objective function constrained to limits on supply and demand of natural gas. The lower or inner level problem is the traditional optimal gas flow (OGF) which minimizes the operating cost of the pipeline network constrained to the physical gas flow dynamics and the pressure and network flow limits.

\subsection{Market Equilibrium problem}
The market equilibrium optimization (MEO) problem maximizes the total economic surplus by minimizing the production/injection cost minus the sales revenue from supplying the natural gas demand. The objective function for the upper problem ($\phi_u$) is written as: 
\begin{equation}\label{eq:F}
    \Phi_u = \sum_{j\in\mathbb G} C_j(g_j) - \sum_{i\in\mathbb D} V_i(d_i)+ \theta c_T\sum_{ij\in\mathbb P}f_{ij},
\end{equation}
where the sets of producer and consumer are denoted by $\mathbb G$ and $\mathbb D$, $C_j(g_j)$ denote producer $j$'s production cost and $V_i(d_i)$ the value of demand served to consumer $i$. Additionally, a cost term on pipe flows in the network is included with parameter $\theta$ and per unit flow cost $c_T$. This term can represent network flow based tariffs or additional maintenance cost of using the pipeline network. Also, there are capacity based limits on the injection or production flows $g_j$ and the withdrawal or consumer flows $d_i$. The complete market equilibrium problem can be formulated as:
\begin{subequations}
\begin{align}\label{eq:upper}
    \mathbf{(P_{upper})} \quad & \min_{g_j,d_i}  &&  \Phi_{u} =  \sum_{j\in\mathbb G} C_j(g_j) - \sum_{i\in\mathbb D} V_i(d_i)+ \theta c_T\sum_{ij\in\mathbb P}f_{ij} \\
    & \text{s.t.} && \underline g_j \leq g_j \leq \overline{g_j}, \\
    & && \underline d_i \leq d_i \leq \overline{d_i}, \\
    & && \sum_j g_j - \sum_i d_i = 0.
\end{align}    
\end{subequations}

Note that without the tariff term, i.e. if $\theta=0$, the market equilibrium model is agnostic to the pipeline network.

\subsection{Optimal Gas Flow (OGF) problem}
The natural gas operations problem or the OGF requires solving an optimization problem with physics based constraints on the pressure and flow rates of natural gas inside the pipes and minimizing the operations cost for compressors. This is done by modeling the natural gas pipeline network as a connected graph with pipe junctions represented by nodes with node-set $\mathbb N$ and arc/edges representing the pipes and compressors with set $\mathbb E=\mathbb P\cup\mathbb C$, where $\mathbb P$ and $\mathbb C$ are the disjoint sets of pipes and compressors.

\subsubsection{Nodal balance constraints}
At each node in the network, the net incoming ($\partial^-$) and outgoing ($\partial^+$) flow is balanced with the total withdrawal of natural gas,
\begin{equation}
    \sum_{i \in \partial^-_j} f_{ij} - \sum_{k \in \partial^+_j} f_{jk} = g_j - d_j, \quad \forall j \in \mathbb{N}.
\end{equation}
The nodal flow balance constraint can be rewritten in vector form using the incidence matrix $A \in \mathbb{R}^{|\mathbb{N}| \times |\mathbb{E}|}$,
\begin{equation}\label{eq:node_balance}
    A \mathbf{f} = B_g \mathbf{g} - B_d \mathbf{d}.
\end{equation}
where the matrices $B_g \in \mathbb R^{|\mathbb{N}| \times |\mathbb{G}|}$ and $B_d \in  \mathbb R^{|\mathbb{N}| \times |\mathbb{D}|}$ place producers and consumers at their nodes. Each column of these two matrices contains one unit entry and is zero otherwise. Note that, the model does not consider the compressor fuel withdrawals and other mass losses.

\subsubsection{Pressure Drop Constraint}
Across each pipe $(i,j) \in \mathbb{P}$, the pressure drop in natural gas flow is given by the steady-state Weymouth equation
\begin{equation}\label{eq:weymouth_exact}
    \Pi_i-\Pi_j=\omega_{ij} f_{ij}|f_{ij}|, \qquad ij\in\mathbb P,
\end{equation}
where $\Pi_i = p^2$ is the squared gas pressure at nodes and $\omega_{ij} = \frac{a^2 \lambda_{ij} L_{ij}}{A^2_{ij}D_{ij}}$ is the pipe resistance coefficient with $a$ denote the wave speed of natural gas. $\lambda, L, A$ and $D$ denote the friction factor, length, cross-section area and diameter of the pipe respectively \cite{srinivasan2026solution}. The Weymouth equation is nonlinear and nonconvex constraint because of the $f|f|$ term.

\subsubsection{Compressor balance}
For compressors $(i,j) \in \mathbb{C}$, the pressure across the compressor is modeled as a simple constraint using the compressor ratio variable $\alpha$ such that:
\begin{equation}\label{eq:compressor}
    \Pi_j - \alpha_{ij} \Pi_i = 0, \qquad ij \in \mathbb{C},
\end{equation}
with bounds on the compressor ratio as $1 \leq \alpha \leq \alpha^{max}$.

\subsubsection{Slack node}
At each slack node in the network, the pressure is fixed to a constant value,
\begin{equation}
    \Pi_s = K,
\end{equation}
where $\Pi_s$ denotes the squared slack pressure.

\subsubsection{Gas-network limit constraints}
Each pipe and compressor edge in the network $e \in \mathbb{E}$ has finite flow capacities. Pipes permit flow in either direction, whereas compressors operate only in their prescribed direction:
\begin{align}
    -\overline{f_{ij}}& \leq f_{ij} \leq \overline{f_{ij}}, &&  ij \in \mathbb{P}, \label{eq:pipe_flow_bounds}\\
    0 & \leq f_{ij} \leq \overline f_{ij}, && ij\in \mathbb{C}. \label{eq:compressor_flow_bounds}
\end{align}
At each node, squared pressure satisfies
\begin{equation}\label{eq:pressure_bounds}
    \underline{\Pi}_i \leq \Pi_i \leq \overline{\Pi}_i, \qquad i \in \mathbb{N},
\end{equation}
where $\underline{\Pi}_i$ and $\overline{\Pi}_i$ are lower and upper pressure limits.

\subsubsection{Compressor Work Objective}
The objective of the optimal gas flow problem is to minimize the operating cost of compressors. The compressor cost is directly proportional to the compressor work,
\begin{equation}\label{eq:compressor_work}
    W_c = \sum_{ij \in C} f_{ij} (\alpha_{ij}^m - 1), 
\end{equation}
where $m$ denotes the isotropic coefficient and $W_c$ denotes the total compressor work done by the compressors \cite{sundar2018state}.

\subsubsection{Problem Formulation}
The complete formulation of the optimal gas flow (OGF) problem is 
\begin{align}
    \mathbf{(P_{lower})} \quad & \min_{\alpha} \quad  \Phi_{l} =  c_e W_c (\alpha) \\
    & \text{s.t.} \quad  \text{Eq.}\eqref{eq:node_balance} -\eqref{eq:compressor_work}.
\end{align}
Here $c_e$ denotes the per-unit electricity cost and $\Phi_l$ is the total operating cost of compressors, which is used as the lower level objective function.

\subsection{Bilevel Problem}
The bilevel gas network market model solves both the market equilibrium optimization (MEO) jointly with the optimal gas flow (OGF) problem.  The upper level variables, i.e., injection or production $g$ and withdrawal or consumer flows $d$, are the parameters in the lower level OGF problem in Eq.\eqref{eq:node_balance}, whereas the lower level variable viz. the pipe flow rates appear in the upper level objective function. The entire bilevel gas network problem can thus be formulated as:
\begin{subequations}\label{eq:bilevel_problem}
\begin{align}
\mathbf{(P)} \quad & \min_{g_j,d_i,\alpha}  &&  \Phi_{u} =  \sum_{j\in\mathbb G} C_j(g_j) - \sum_{i\in\mathbb D} V_i(d_i)+ \theta c_T\sum_{j\in\mathbb P}f_j \\
& \text{s.t.} && \text{Eq.}(2b)-(2c), \\
& && \alpha = \text{argmin}_{\alpha} \quad \Phi_l = c_e W_c (\alpha), \\
& && \quad \quad \text{s.t.} \quad \qquad \text{Eq.}(4)-(11).
\end{align}   
\end{subequations}
This bilevel problem formulation exactly represents the interplay between the natural gas market and the pipeline network operations.

\section{Convex Relaxations}\label{sec:relax}
Bilevel optimization problems are structurally non-convex and NP-hard \cite{dempe2020bilevel}, making it very difficult to find their optimal solutions. Typical methods reformulate the bilevel problem into a single level optimization problem using either the optimality or duality conditions on the lower level problem. It is known that optimal solutions of the single level reformulation need not be optimal for the original bilevel problem \cite{dempe2020bilevel}. However, it has been shown that if the lower level problem is convex and the global solution of the single level reformulation lies in the interior of it's feasible region, then the solution is a global optimal point to the bilevel problem \cite{DempeDutta}.

The original bilevel gas network market optimization problem formulated in Eq. \eqref{eq:bilevel_problem} is both nonlinear and nonconvex. We thus propose convex relaxations to the lower level problem such that the leader-follower formalism from \cite{DempeDutta} is applicable in this case. The main non-convexities in the lower level problem arise from the objective function, the bilinear compressor equalities, and the nonlinear Weymouth equation. In the following subsections, we propose convex reformulations for the objective function which substitutes the compressor ratios, and relax the Weymouth equalities using a set of convex hull inequalities for fixed directional flow.

\subsection{Convex Objective}
In the lower level OGF problem, the objective is to minimize the operating cost of compressors using the adiabatic compression work expression, which is a non-convex function of the compressor ratio $\alpha$ (see Eq.\eqref{eq:compressor_work}). Alternatively, we can also use a convex surrogate for the compressor work in the form of pressure drop across compressors,
\begin{equation}\label{eq:surr_compressor}
    \tilde \Phi_l = \beta_c \sum_{ij \in C} (\Pi_j - \Pi_i),
\end{equation}
where $\beta_c$ can be used as a scaling parameter to adjust the magnitude of surrogate function with the nonconvex compressor work. This is based on the fact that gas pipeline operators explicitly control the compressor discharge pressure ($\Pi_{j}$) and try to minimize the pressure rise i.e. difference between inlet and outlet pressures as shown in Eq. \eqref{eq:surr_compressor}.

\subsection{Equivalent compressor constraints}
The compressor equality \eqref{eq:compressor} is bilinear
in the decision variables $\alpha_{ij}$ and $\Pi_i$. After replacing the compressor-work objective by Eq. 
\eqref{eq:surr_compressor}, the ratio variables appear only in this equality and their bounds. Since operating pressures satisfy $\Pi_i>0$, eliminating $\alpha_{ij}$ gives the equivalent affine constraints

\begin{equation}\label{eq:compressor_affine}
\begin{cases}
  \Pi_i-\Pi_j &\leq 0,\quad  ij\in\mathbb C, \\
\Pi_j-\alpha^{\max}\Pi_i &\leq 0, \quad ij\in\mathbb C.
\end{cases}
\end{equation}
Conversely, every pressure pair satisfying Eq. \eqref{eq:compressor_affine} defines an admissible ratio
$\alpha_{ij}=\Pi_j/\Pi_i\in[1,\alpha^{\max}]$. Thus, this elimination preserves the feasible
flow-pressure pairs and introduces no additional
relaxation. Hence the follower variables can therefore be
written as $y=(f,\Pi)$.

\subsection{Convexified Weymouth Equation}

We obtain the convex relaxation of the nonlinear Weymouth equation, through two distinct modeling steps: (i) selecting an operating direction for each pipe and (ii) relaxing the corresponding Weymouth equality. 

\subsubsection{Fixed-direction assumption}
In the exact model, a pipe ($ij \in \mathbb{P}$) may carry gas in either direction. Its reference orientation defines $f_{ij}>0$ as flow from $i$ to $j$ and $f_{ij}<0$ as flow from $j$ to $i$.

We now assume that each pipe operates only in its reference orientation and impose
\begin{equation} \label{eq:fixed_flow_bounds}
 0\leq f_{ij}\leq\overline{f_{ij}},
 \qquad ij\in\mathbb{P}.
\end{equation}
The permitted direction of each pipe is chosen
during model construction and held fixed throughout optimization. The solver determines the nonnegative flow rate in that specified direction. On this flow interval, $f_{ij}|f_{ij}|=f_{ij}^2$ term in the Weymouth equation becomes
\begin{equation}\label{eq:weymouth_fixed}
 \Pi_i-\Pi_j=\omega_{ij} f_{ij}^2,
 \qquad ij\in\mathbb{P}.
\end{equation}

This restriction selects a single quadratic branch of the signed Weymouth relation over a bounded flow interval. However, the equality graph in Eq. \eqref{eq:weymouth_fixed} remains nonconvex; selecting a direction alone does not convexify the pipe constraint. For a fixed market allocation $x$, let
$\mathbb Y_{\mathrm{fd}}(x)$ denote the exact network feasible set with the additional nonnegative pipe-flow restrictions. Then
\begin{equation} \label{eq:direction_restriction}
 \mathbb Y_{\mathrm{fd}}(x)
 \subseteq
 \mathbb Y_{\mathrm{ex}}(x).
\end{equation}
Thus, fixing directions restricts the exact model rather than relaxes it.  Market allocations that are feasible with bidirectional pipe flows may become infeasible under the selected operating directions.

\subsubsection{Convex relaxation of pipe constraints}
Figure \ref{fig:envelope}(a) shows the exact bidirectional Weymouth relation. The direction restriction in Eq. 
\eqref{eq:fixed_flow_bounds} selects its nonnegative branch, $\Pi_i-\Pi_j=\omega_{ij} f_{ij}^2$, whose equality graph remains nonconvex. We now replace this graph by its convex envelope over the prescribed flow interval. For a pipe ($ij\in\mathbb{P}$), we write $\Delta\Pi_{ij}=\Pi_i-\Pi_j$. On the interval $0\le f_{ij}\le\overline{f_{ij}}$, the fixed-direction
Weymouth relation is the graph of the convex function $\Delta\Pi_{ij}=\omega_{ij} f_{ij}^2$. The straight line joining its endpoints $(0,0)$ and
$(\overline f_{ij},\omega_{ij}\overline f_{ij}^2)$ is
$\Delta\Pi_{ij}=\omega_{ij}\overline{f_{ij}} f_{ij}$. Since $f_{ij}^2\leq \overline{f_{ij}} f_{ij}$ for $0\leq f_{ij}\leq \overline{f_{ij}}$, the quadratic Weymouth curve lies below this chord. We replace the equality by
\begin{equation}\label{eq:envelope}
 \omega_{ij} f_{ij}^2
 \le \Pi_i-\Pi_j
 \le \omega_{ij}\overline{f_{ij}} f_{ij},
\end{equation}
while retaining the flow bounds in
\eqref{eq:fixed_flow_bounds}. 

Together, these inequalities describe the convex hull of the isolated forward-flow graph, shown as the shaded region in Fig. \ref{fig:envelope}(b). The marked interior point illustrates that satisfying the relaxation does not imply satisfying the Weymouth equality.
\begin{figure}[thpb]
 \centering
 \framebox{%
  \parbox{3in}{%
   \centering
   \includegraphics[width=\linewidth]{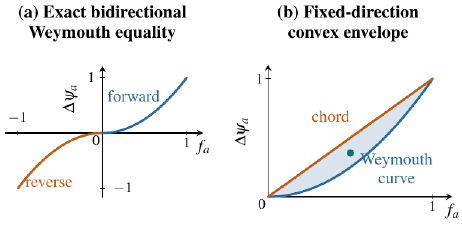}%
  }%
 }
 \caption{Single-pipe Weymouth relation and its
 fixed-direction relaxation.
 (a) Exact bidirectional equality.
 (b) Convex envelope of the selected nonnegative branch; the marked interior point violates the Weymouth equality.
 Here $\Delta\Pi_{ij}=\Pi_i-\Pi_j$, and the plots use $\omega_{ij}=\overline f_{ij}=1$.}
 \label{fig:envelope}
\end{figure}
The quadratic inequality in \eqref{eq:envelope}
admits the rotated second-order-cone representation
\begin{equation}\label{eq:soc}
 \left(
 \Pi_i-\Pi_j,
 \frac{1}{2},
 \sqrt{\omega_{ij}}f_{ij}
 \right)
 \in\mathbb Q_r^3,
\end{equation}
where
\begin{equation*}
 \mathbb Q_r^3
 =
 \left\{
 (r,s,t)\in\mathbb R^3:
 r,s\geq 0,\; 2rs\geq t^2
 \right\}.
\end{equation*}
Substitution into the cone inequality gives
$\Pi_i-\Pi_j\ge\omega_{ij} f_{ij}^2$.
The upper inequality in \eqref{eq:envelope} is affine. Combining the pipe envelopes \eqref{eq:envelope} with the affine compressor constraints
\eqref{eq:compressor_affine}, nodal balances, operating bounds, and prescribed slack pressure gives a convex feasible set of each fixed market allocation. For the relaxed model, let $x=(g,d)$ denote the upper-level market allocation and $y=(f,\Pi)$ the lower-level decision variables. The allocation $x$ is fixed when solving the lower-level problem. Define the equality mapping by
\begin{equation}\label{eq:equality_map}
    H(x,y):=\begin{bmatrix}
        Af-B_g g+B_d d\\
        \Pi_s-K
    \end{bmatrix}.
\end{equation}
Thus, $H(x,y)=0$ imposes all nodal flow balances and the prescribed squared pressure at slack node $s$. Similarly, for the mapping 
\begin{equation}\label{eq:inequality_map}
    G(y):=
\begin{bmatrix}
  \bigl(\omega_{ij}f_{ij}^{2}-\Pi_i+\Pi_j\bigr)_
    {ij\in\mathbb P}\\
  \bigl(\Pi_i-\Pi_j-\omega_{ij}\overline f_{ij}f_{ij}\bigr)_
    {ij\in\mathbb P}\\
  \bigl(-f_{ij}\bigr)_{ij\in\mathbb E}\\
  \bigl(f_{ij}-\overline f_{ij}\bigr)_{ij\in\mathbb E}\\
  \bigl(\underline\Pi_i-\Pi_i\bigr)_{i\in\mathbb N}\\
  \bigl(\Pi_i-\overline\Pi_i\bigr)_{i\in\mathbb N}\\
  \bigl(\Pi_i-\Pi_j\bigr)_{ij\in\mathbb C}\\
  \bigl(\Pi_j-\alpha^{\max}\Pi_i\bigr)_{ij\in\mathbb C}
\end{bmatrix},
\end{equation}
imposes all inequality constraints, where each indexed block denotes a column vector. The condition $G(y) \leq 0$ imposes the pipe-envelopes, non-negative arc-flow bounds, pressure bounds, and affine compressor constraints. Then for fixed $x$, $H(x, \cdot)$ is affine and each component of $G$ is convex. With $Q(y):=\tilde\Phi_l(y)$ from Eq. \eqref{eq:surr_compressor}, the relaxed follower response set is
\begin{align}\label{eq:S_rel}
    \mathbb S_{\mathrm{rel}}(x)
    &:= \operatorname*{arg\,min}_{y}\quad
     Q(y)\\
    \text{s.t.}\quad
    & H(x,y)=0,
    \nonumber\\
    & G(y)\leq 0.
    \nonumber
\end{align}
Replacing the exact follower response set in Eq. 
\eqref{eq:bilevel_problem} by
$\mathbb S_{\mathrm{rel}}(x)$ gives
\begin{align}\label{eq:bilevel_rel}
    \min_{x,y}\quad
    & F_\theta(x,y)\\
    \text{s.t.}\quad
    & x\in\mathbb X,
    \nonumber\\
    & y\in\mathbb S_{\mathrm{rel}}(x).
    \nonumber
\end{align}

To obtain a single-level optimization problem, we replace the response constraint $y\in\mathbb S_{\mathrm{rel}}(x)$ by explicit conditions characterizing the follower optimality. The next section introduces four representations: the KKT-based MPCC, Wolfe-duality reformulation (WDP), Mond-Weir reformulation (MDP), extended Mond-Weir reformulation (eMDP), and Lagrange-dual reformulation (LDR).

\section{Single-Level Reformulations}\label{sec:reform}
Let $m$ denote the number of inequalities in $G(y)$. Associate multipliers $u\in\mathbb R_{\geq0}^{m}$ with $G(y)\leq0$ and unrestricted multipliers
$v\in\mathbb R^{|\mathbb N|+1}$ with $H(x,y)=0$. Then, the follower Lagrangian, dual function, and dual value are
\begin{align}\label{eq:lower_lagrangian}
 \boldsymbol{\mathit{L}}(x,y,u,v) &:=Q(y)+u^{\mathsf T}G(y)+v^{\mathsf T}H(x,y),\\
 d(x,u,v)&:=\inf_{z}~\boldsymbol{\mathit{L}}(x,z,u,v),
 \quad
 d^*(x):=\sup_{u\geq0,\,v}d(x,u,v),
 \nonumber
\end{align}
where the infimum runs over $z\in\mathbb R^{|\mathbb A|+|\mathbb N|}$ and may equal $-\infty$. Let $\varphi(x)$ denote the optimal value of the relaxed follower problem in Eq. \eqref{eq:S_rel}.
Derivatives with respect to the network variables treat the market allocation $x$ as fixed.

The equivalence statements below are conditional on the following assumption: for every $x\in\mathbb X$ with a nonempty relaxed follower feasible set, its Lagrangian dual attains the follower's optimal value. This assumption ensures the existence of KKT multipliers at follower minimizers. Convexity, in turn, makes the KKT
conditions sufficient for global follower
optimality \cite{boyd2004convex}.

Each formulation minimizes $F_\theta(x,y)$ over the physical variables and its certificate variables, subject to
\begin{equation}\label{eq:common_primal}
    x\in\mathbb X,\qquad
    H(x,y)=0,\qquad
    G(y)\leq0,
\end{equation}
and the additional conditions specified below.

All three dual-based formulations introduce an
auxiliary point $z=(f^z,\Pi^z)\in\mathbb R^{|\mathbb A|+|\mathbb N|}$ and require
\begin{equation}\label{eq:dual_stationarity}
    \nabla_z \boldsymbol{\mathit{L}} (x,z,u,v)=0.
\end{equation}
The point $z$ is used to certify follower optimality; it is not required to satisfy $G(z)\leq0$ or $H(x,z)=0$. The same market allocation $x$ appears in the physical constraints and the certificate conditions \cite{li2026comparison}.

\begin{itemize}[leftmargin=*]
\item \textbf{KKT-based reformulation (MPCC):} The MPCC supplements the common constraints \eqref{eq:common_primal} with follower stationarity, dual feasibility, and complementarity:
 \begin{align}\label{eq:mpcc_certificate}
    &\nabla_y \boldsymbol{\mathit{L}}(x,y,u,v)=0,\\
    &u_\ell G_\ell(y)=0,
    \qquad \ell=1,\ldots,m.
    \nonumber
 \end{align}
Together with $u\geq0$ and $G(y)\leq0$, the second condition imposes the complementarity pairs $0\leq u_\ell\perp-G_\ell(y)\geq0$. Complementarity requires a strictly slack
inequality to have zero multiplier; an active
inequality may have a nonzero multiplier.

\item \textbf{Wolfe-duality reformulation (WDP):} WDP introduces the auxiliary network variables $z$ and the multipliers $(u,v)$.
It combines the stationarity condition
\eqref{eq:dual_stationarity} with the Wolfe-gap inequality
\begin{align}\label{eq:wdp_certificate}
    Q(y)-Q(z)&-u^{\mathsf T}G(z) -v^{\mathsf T}H(x,z)\leq 0.
\end{align}
Stationarity and convexity make $\boldsymbol{\mathit{L}}(x,y,u,v)$ a lower bound on the follower's optimal value. The gap inequality requires the feasible response $y$ to attain that bound.

\item \textbf{Mond-Weir reformulation (MDP):}
MDP retains Eq. \eqref{eq:dual_stationarity} and
replaces the combined Wolfe-gap inequality by
  \begin{align}\label{eq:mdp_certificate}
    &Q(y)-Q(z)\leq 0,\\
    &u^{\mathsf T}G(z)+v^{\mathsf T}H(x,z) \geq 0.
    \nonumber
\end{align}
That is, these conditions separate the follower-objective comparison from an aggregate condition on the multiplier-weighted constraints.

\item \textbf{Extended Mond--Weir reformulation
(eMDP):} eMDP also retains Eq. \eqref{eq:dual_stationarity} and imposes
  \begin{align}\label{eq:emdp_certificate}
    &Q(y)-Q(z) \leq 0,\\
    &u_\ell G_\ell(z) \geq 0,
    \qquad \ell=1,\ldots,m,
    \nonumber\\
    &v_i H_i(x,z)=0,
    \qquad i\in\mathbb N.
    \nonumber
\end{align}
Notice that, the aggregate MDP condition is replaced by componentwise conditions. Unlike MPCC complementarity, the inequality products are required to be nonnegative rather than zero.

\item \textbf{Lagrange-dual reformulation (LDR)}: The LDR formulation uses the dual function in Eq. \eqref{eq:lower_lagrangian} by considering common constraints \eqref{eq:common_primal} and the following single inequality for $u\geq 0$, 
\begin{equation}\label{eq:ldr_certificate}
    Q(y)\leq d(x,u,v).
\end{equation}
Because $d$ is an infimum, Eq. \eqref{eq:ldr_certificate} is the infinite family $Q(y)\leq\boldsymbol{\mathit{L}}(x,z,u,v)$ for all $z\in\mathbb R^{|\mathbb A|+|\mathbb N|}$, so LDR is a semi-infinite program with a noncompact index set. Such problems are computationally expensive in general, because evaluating $d(x,u,v)$ for a nonconvex lower level is itself a global optimization problem \cite{turan2025bounding,mehrotra2014cutting,melcher2026semi}. In our convex setting, the inner infimum is attained exactly at the stationary points of $\boldsymbol{\mathit{L}}(x,\cdot,u,v)$, which is what the finite formulations above exploit.
\end{itemize}
Assuming that every relaxed follower minimizer
admits KKT multipliers, then all five formulations
have the same feasible projection onto $(x,y)$
as $y\in\mathbb S_{\mathrm{rel}}(x)$. For any follower minimizer, its KKT multipliers satisfy the MPCC conditions, and setting $z=y$ also satisfies each dual-based formulation. Since $y$ is then a stationary point of the convex Lagrangian $\boldsymbol{\mathit{L}}(x,\cdot,u,v)$ and complementarity gives $\boldsymbol{\mathit{L}}(x,y,u,v)=Q(y)$, the same multipliers yield $d(x,u,v)=Q(y)$ and hence satisfy the LDR inequality \eqref{eq:ldr_certificate} as well. Nevertheless, the resulting five single-level formulations remain nonconvex, so a local numerical solution does not establish global optimality of the leader objective.

The literature shows that MPCC, WDP, MDP, eMDP, and LDR characterize the feasible set of Eq. \eqref{eq:bilevel_rel} when strong duality holds with dual attainment for every $x$ \cite{li2023wolfe,li2024mondweir,dempe2025duality,li2026comparison}. In this section, we make two extensions: First, without assuming strong duality or dual attainment, we show that all five formulations project onto the same set of the certifiable part of the follower response graph. Thus, none can generate a spurious follower response, and the assumption is needed only to guarantee that every follower-optimal response can be represented in the corresponding reformulation. Second, when the objective comparison in a dual-based formulation is enforced only within a tolerance $\varepsilon$, as in numerical computations, we show that $\varepsilon$ simultaneously bounds the follower suboptimality of $y$, the duality gap at $x$, and the multiplier suboptimality. We also characterize the smallest tolerance for which a given response can be lifted.

For $\varepsilon\geq0$, the $\varepsilon$-relaxed WDP, MDP, and eMDP replace the objective comparison in 
Eqs. \eqref{eq:wdp_certificate}, \eqref{eq:mdp_certificate}, and \eqref{eq:emdp_certificate} by $Q(y)-\boldsymbol{\mathit{L}}(x,z,u,v)\leq\varepsilon$ and $Q(y)-Q(z)\leq\varepsilon$, respectively. The $\varepsilon$-relaxed LDR replaces Eq. \eqref{eq:ldr_certificate} by $Q(y)\leq d(x,u,v)+\varepsilon$. All other conditions remain unchanged.
\begin{theorem}\label{thm:exactness}
Let $Q$ and $G$ be convex and continuously differentiable and $H$ affine in $y$ and define $\mathbb X_{\mathrm D}:=\{x\in\mathbb X:\ d(x,u,v)=\varphi(x)\ \text{for some } u\geq0,\ v\}$.
\begin{enumerate}[label=(\roman*), leftmargin=*]
\item The feasible sets of MPCC, WDP, MDP, eMDP, and LDR have the same projection onto $(x,y)$, namely
\begin{equation}\label{eq:projection}
    \{(x,y):\ x\in\mathbb X_{\mathrm D},\ y\in\mathbb S_{\mathrm{rel}}(x)\}.
\end{equation}
In particular, every feasible point of any of the five formulations is feasible for \eqref{eq:bilevel_rel}, and the projection equals the feasible set of \eqref{eq:bilevel_rel} if and only if $\mathbb X_{\mathrm D}=\mathbb X$.
\item Every feasible point of an $\varepsilon$-relaxed WDP, MDP, eMDP, or LDR satisfies
\begin{equation}\label{eq:three_term}
    \underbrace{Q(y)-\varphi(x)}_{\text{follower subopt.}}
    +\underbrace{\varphi(x)-d^*(x)}_{\text{duality gap}}
    +\underbrace{d^*(x)-d(x,u,v)}_{\text{multiplier subopt.}}
    \leq\varepsilon,
\end{equation}
and each of the three terms is nonnegative.
\item Let $x\in\mathbb X$, $y\in\mathbb S_{\mathrm{rel}}(x)$, and $\tau(x):=\varphi(x)-d^*(x)$. Then $(x,y)$ lifts to a feasible point of the $\varepsilon$-relaxed LDR for every $\varepsilon>\tau(x)$ and for no $\varepsilon<\tau(x)$; at $\varepsilon=\tau(x)$ a lift exists if and only if the supremum defining $d^*(x)$ is attained. The same holds for the $\varepsilon$-relaxed WDP whenever the infimum defining $d(x,u,v)$ is attained at the lifting multipliers, for instance when $Q$ is strongly convex. If $x\in\mathbb X_{\mathrm D}$, then $(x,y)$ lifts to every relaxed formulation for every $\varepsilon\geq0$.
\end{enumerate}
\end{theorem}
\begin{proof}
(i) Let a point be feasible for one of the five formulations. Feasibility of $y$ and weak duality give $Q(y)\geq\varphi(x)\geq d^*(x)\geq d(x,u,v)$. Next, we show $Q(y)\leq d(x,u,v)$. For LDR this is trivial from Eq. \eqref{eq:ldr_certificate}. For WDP, MDP, and eMDP, stationarity \eqref{eq:dual_stationarity} and convexity of $\boldsymbol{\mathit{L}}(x,\cdot,u,v)$ give $d(x,u,v)=\boldsymbol{\mathit{L}}(x,z,u,v)$, and Eq. \eqref{eq:wdp_certificate} then implies $Q(y)\leq d(x,u,v)$; the MDP conditions sum to Eq. \eqref{eq:wdp_certificate}, and the eMDP conditions sum to the MDP conditions. For MPCC, complementarity and $H(x,y)=0$ give $\boldsymbol{\mathit{L}}(x,y,u,v)=Q(y)$, and stationarity at $y$ gives $d(x,u,v)=\boldsymbol{\mathit{L}}(x,y,u,v)$. Hence, the claim holds and we obtain $y\in\mathbb S_{\mathrm{rel}}(x)$ and $x\in\mathbb X_{\mathrm D}$. Conversely, let $x\in\mathbb X_{\mathrm D}$, $y\in\mathbb S_{\mathrm{rel}}(x)$, and $(u,v)$ with $d(x,u,v)=\varphi(x)=Q(y)$. Then
\[
    d(x,u,v)\leq\boldsymbol{\mathit{L}}(x,y,u,v)=Q(y)+u^{\mathsf T}G(y)\leq Q(y)=d(x,u,v),
\]
so $u^{\mathsf T}G(y)=0$, which gives $u_\ell G_\ell(y)=0$ for every $\ell$ because each product is nonpositive, and $y$ minimizes $\boldsymbol{\mathit{L}}(x,\cdot,u,v)$, which gives $\nabla_y\boldsymbol{\mathit{L}}(x,y,u,v)=0$. Thus $(x,y,u,v)$ is MPCC feasible, and $z=y$ makes it feasible for WDP, MDP, eMDP, and LDR. The last claim follows because $\mathbb S_{\mathrm{rel}}(x)$ is nonempty for every $x\in\mathbb X$.
 
(ii) The argument of (i) applied to the relaxed conditions gives $Q(y)-d(x,u,v)\leq\varepsilon$. The three terms in Eq. \eqref{eq:three_term} are nonnegative by feasibility of $y$ and weak duality, and they telescope to $Q(y)-d(x,u,v)$.
 
(iii) Because $Q(y)=\varphi(x)$, a lift to the relaxed LDR exists if and only if $d(x,u,v)\geq\varphi(x)-\varepsilon=d^*(x)+\tau(x)-\varepsilon$ for some $(u,v)$. If $\varepsilon>\tau(x)$, the definition of the supremum provides such $(u,v)$. If $\varepsilon<\tau(x)$, no $(u,v)$ qualifies because $d(x,u,v)\leq d^*(x)$. If $\varepsilon=\tau(x)$, the requirement is $d(x,u,v)=d^*(x)$. For WDP, let $z$ attain the infimum defining $d(x,u,v)$ at the lifting multipliers. Then $z$ satisfies Eq. \eqref{eq:dual_stationarity} by convexity, and $Q(y)-\boldsymbol{\mathit{L}}(x,z,u,v)=Q(y)-d(x,u,v)\leq\varepsilon$ is the relaxed WDP inequality. Strong convexity of $Q$ makes $\boldsymbol{\mathit{L}}(x,\cdot,u,v)$ strongly convex, so its infimum is attained. Finally, if $x\in\mathbb X_{\mathrm D}$, the exact lift constructed in (i) is feasible for every relaxed formulation at every $\varepsilon\geq0$.
\end{proof}

Theorem \ref{thm:exactness} has two consequences from the computational perspective. First, any local solution of the finite formulations yields a follower-optimal response even at market allocations where strong duality is not known to hold; therefore, no follower constraint qualification must be verified in advance. Second, for a computed WDP, MDP, or eMDP solution whose stationarity condition \eqref{eq:dual_stationarity} holds to solver precision, the objective-comparison residual $\varepsilon$ provides, by Eq. \eqref{eq:three_term}, a certified upper bound on the follower suboptimality $Q(y)-\varphi(x)$ without additional assumptions. However, MPCC provides no analogous bound because a complementarity residual does not control $Q(y)-\varphi(x)$ without additional error-bound constants.

\section{Numerical Experiments}\label{sec:numerics}\label{sec:numer}
We compare MPCC, WDP, MDP, and eMDP on the six-node network in Fig. \ref{fig:network}, using both exact
bidirectional pipe physics and the fixed-direction convex relaxation. Consumer data are given in Table \ref{tab:demand}, with total requested demand
$\sum_{i\in\mathbb D}\overline d_i=130$ kg/s. Production is bounded by $1000$ kg/s. Nodal pressures lie between $3$ and $6$ MPa, with the slack pressure prescribed as $p_1=4$ MPa. Pressures are normalized by $3$ MPa before squaring, giving the numerical slack value $\Pi_1=16/9$. Both experiments maximize $-\Phi_u$ in Eq. \eqref{eq:F}, with $\theta=0.05$, $c_T=1$, and a signed pipe flow tariff. For these experiments, we use the convex surrogate compressor-work objective in \eqref{eq:surr_compressor} and \eqref{eq:S_rel} with $\beta_c=1$, that is
\begin{equation}\label{eq:numerical_surrogate}
    Q(y)=\sum_{(i,j)\in C}(\Pi_j-\Pi_i).
\end{equation}
The same surrogate is used in both models, preserving convexity of the relaxed follower.

\begin{table}[!t]
\begin{center}
\begin{tabular}{|c||c||c||c|}
\hline
Consumer & Node & $\overline d_i$ & $b_i$\\
\hline
1 & 2 & 30 & 3.0\\
2 & 3 & 40 & 4.0\\
3 & 4 & 20 & 5.0\\
4 & 3 & 30 & 2.5\\
5 & 4 & 10 & 3.0\\
\hline
\end{tabular}
\caption{Consumer data: $\overline d_i$ is in kg/s and $V_i(d_i)=b_i d_i$.}\label{tab:demand}
\end{center}
\end{table}

\begin{figure}[!t]
 \centering
 \framebox{%
  \parbox{3.0in}{%
   \centering\includegraphics[width=\linewidth]{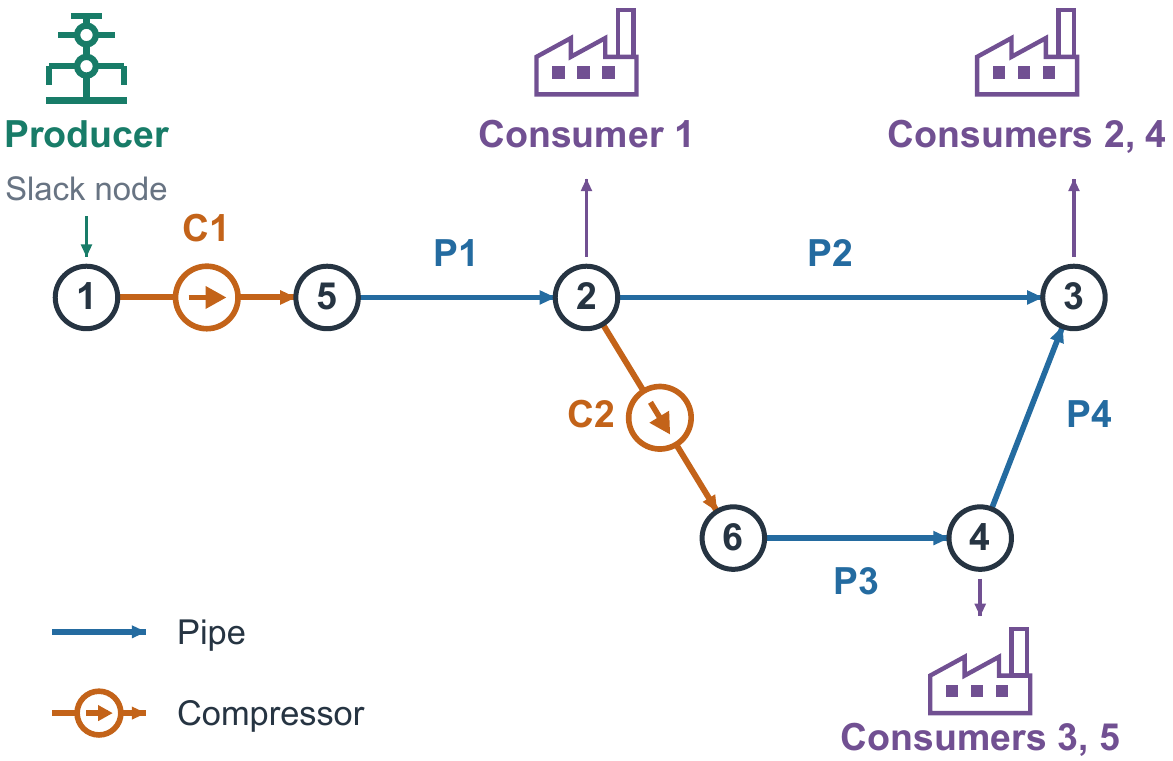}%
  }%
 }
    \caption{Six-node network. Blue arrows indicate pipe reference orientations; orange connections denote directed compressors. Pipe flows may reverse in the exact model but not in the relaxation.}\label{fig:network}
\end{figure}

\subsection{Solvers and reference construction}
The experiments use Python 3.12.14, Ipopt 3.14.19, and Gurobi 12.0.2. Ipopt uses exact Hessians, stopping and constraint-violation tolerances of $10^{-8}$, and a budget of $5600$ iterations and $120$ seconds per configuration. MPCC shares this budget across Scholtes parameters $10^{-2},10^{-3},\ldots,10^{-8}$. Gurobi uses feasibility and optimality tolerance of $10^{-9}$, exact model and relative gap tolerances of $10^{-6}$. Reference verification also uses physical feasibility and follower gap thresholds of $10^{-6}$ and an upper gap threshold of $10^{-6} \max\{1, |\Phi_u|\}$. All Ipopt configurations use a common prescribed physical starting point at zero, while MPCC imposes $0\leq u_\ell[-G_\ell(y)]\leq t$ at each Scholtes stage, retaining primal and dual feasibility. For Gurobi's exact-network model, the signed Weymouth equation is represented by $r_{ij}\geq 0,\quad r_{ij}^2=f_{ij}^2, \quad \Pi_i-\Pi_j=\omega_{ij} f_{ij} r_{ij}, \quad ij \in\mathbb P.$ Thus, $r_{ij}=|f_{ij}|$ without restricting the sign of $f_{ij}$. The exact dual-based models are experimental since the convex-follower equivalence argument for the dual-based reformulations is not assumed to hold for the exact follower \cite{li2026comparison}. Gurobi provides a physical-network reference and a numerical lower bound for each model. A separate follower solve fixes the reference market allocation and checks its original network response.

\subsection{Results}
We report relative differences in both objectives and total served demand relative to the corresponding best-known Gurobi reference. Network feasibility and follower stationarity are measured through the mass-balance $(R_B)$, Weymouth $(R_W)$, and projected follower stationarity residuals $R_S=\|y-\Pi_{\mathbb Y_{{op}}}(y-s_y)\|_2$ for the exact and relaxed models, where $\Pi_{\mathbb Y_{{op}}}$ denotes Euclidean projection onto the flow and pressure bounds, compressor limits, and
prescribed slack-pressure constraint, and $s_y$ contains the follower-objective gradient and the multiplier contributions from the balance and pipe constraints. Table \ref{tab:exact_relaxed_results} reports percentage changes in the upper and lower objectives ($\Delta_{\Phi_u}$, $\Delta_{Q}$) and the percentage shortfall in total served demand ($\Delta_D$), using the following formulas
\begin{align*}
\Delta\Phi_u &= 100\frac{\Phi_u-\Phi_u^*}{|\Phi_u^*|},\\
\Delta Q &= 100\frac{Q-Q^*}{|Q^*|},
\qquad
\Delta D = 100\frac{D^*-D}{D^*},
\end{align*}
where starred quantities are evaluated at the best-known solution obtained using Gurobi. Table \ref{tab:exact_relaxed_residuals} reports $\Delta R_k=\max\{0,R_k-\tau_k\}$,
$k\in\{B,W,S\}$.

\begin{table}[!t]
\begin{center}
\small
\setlength{\tabcolsep}{3pt}
\begin{tabular}{|c||c||c||c|}
\hline
Method & $\Delta \Phi_u $ & $\Delta Q $ & $\Delta D$\\
\hline
\multicolumn{4}{|c|}{Convex-relaxed follower}\\
\hline
MPCC       & $ 0 \% $  & $ 0 \% $  & $ 0 \% $\\
WDP        & $ 0 \% $  & $ 0 \% $  & $ 0 \% $\\
MDP        & $ 0 \% $  & $ 0 \% $  & $ 0 \% $\\
eMDP$^{*}$ & $58.1 \%$ & $100.0 \%$ & $70.1 \%$\\
\hline
\multicolumn{4}{|c|}{Exact nonconvex follower}\\
\hline
MPCC       & $ 0 \% $  & $ 0 \% $   & $ 0 \% $\\
WDP        & $ 0 \% $  & $ 0 \% $   & $ 0 \% $\\
MDP        & $1.3 \%$  & $52.4 \%$ & $3.4 \%$\\
eMDP$^{*}$ & $37.3 \%$ & $43.1 \%$ & $36.2 \%$\\
\hline
\end{tabular}
\caption{\scriptsize Percentage differences from the best-known Gurobi solution. For both models, these solutions have $-\Phi^*_u=241.952938$, $Q^*=4.669637$, and total served demand of $D^* = 123.682193$ kg/s out of $130$ kg/s requested. An asterisk marks an unresolved solver run in the mentioned budget.}\label{tab:exact_relaxed_results}
\end{center}
\vspace{-3ex}
\end{table}

\begin{table}[h]
\begin{center}
\small
\setlength{\tabcolsep}{3pt}
\begin{tabular}{|c||c||c||c|}
\hline
Method & $\Delta R_B$ & $\Delta R_W$ & $\Delta R_S$\\
\hline
\multicolumn{4}{|c|}{Convex-relaxed follower}\\
\hline
MPCC       & $ < 10^{-6} $ & $ < 10^{-6} $ & $ < 10^{-6} $\\
WDP        & $ < 10^{-6} $ & $ < 10^{-6} $ & $ < 10^{-6} $\\
MDP        & $ < 10^{-6} $ & $ < 10^{-6} $ & $ < 10^{-6} $\\
eMDP$^{*}$ & $ < 10^{-6} $ & $0.279$ & $ < 10^{-6} $\\
\hline
\multicolumn{4}{|c|}{Exact nonconvex follower}\\
\hline
MPCC       & $ < 10^{-6} $ & $ < 10^{-6} $ & $ < 10^{-6} $\\
WDP        & $ < 10^{-6} $ & $ < 10^{-6} $ & $2.16$\\
MDP        & $ < 10^{-6} $ & $ < 10^{-6} $ & $ < 10^{-6} $\\
eMDP$^{*}$ & $ < 10^{-6} $ & $ < 10^{-6} $ & $1.22$\\
\hline
\end{tabular}
\caption{Residual excesses above
$\tau_B=\tau_W=\tau_S=10^{-6}$ in the respective residual units. $\Delta R_B$ is in kg/s; $\Delta R_W$ uses squared per-unit pressure. "$ - $" indicates a residual value less than the reported tolerance.}\label{tab:exact_relaxed_residuals}
\end{center}
\end{table}

Tables \ref{tab:exact_relaxed_results} and \ref{tab:exact_relaxed_residuals} show that relaxed MPCC, WDP, and MDP match the Gurobi reference objective and service to the reported tolerance with small residuals. For both the convex-relaxed follower and the exact nonconvex follower, the best Gurobi solutions have $-\Phi^*_u=241.952938$, $Q^*=4.669637$, and total served demand of $D^* = 123.682193$ kg/s out of $130$ kg/s requested. Table \ref{tab:exact_relaxed_results} shows the percentage differences with respect to the best-known Gurobi solution using four reformulations. Table \ref{tab:exact_relaxed_residuals} shows residual excesses above $\tau_B=\tau_W=\tau_S=10^{-6}$ in the respective residual units.  For exact physics, MPCC also reproduced the reference values with samll residuals, while WDP matches its objective and service, but has a larger projected stationarity residual. Moreover, MDP returns small residuals at an allocation with lower service and a less favorable upper objective for exact model. The relaxed eMDP run returns a point with a large first-order optimality residual for the complete single-level problem, whereas the exact run reaches the iteration limit.

\section{Conclusions}\label{sec:conc}
We developed a bilevel gas-market model coupling market equilibrium model with network operation using a leader-follower reformulation approach. For the convexified follower, we characterized the common certifiable response set of five single-level formulations and bounded follower suboptimality through dual-certificate tolerances. On the six-node instance, relaxed MPCC, WDP, and MDP, and exact MPCC, matched the best-known Gurobi objectives and service with residuals below $10^{-6}$. Exact WDP matched these values but exceeded the stationarity threshold. These numerical observations do not extend the convex follower guarantees to the exact nonconvex model. Extending the theoretical analysis to this setting is left for future work.


\bibliographystyle{unsrt}
\bibliography{references}

\end{document}